\documentclass{article}
\usepackage[utf8]{inputenc}

\usepackage{amsfonts}
\usepackage{amsmath}
\usepackage{amssymb}
\usepackage{amsthm}
\usepackage{mathtools}
\usepackage{cleveref}

\newtheorem{theorem}{Theorem}
\newtheorem{lemma}[theorem]{Lemma}

\title{Solubility of Additive Forms of Twice Odd Degree over Totally Ramified Extensions of $\mathbb{Q}_2$}
\author{Drew Duncan}
\date{\today}

\begin{document}

\maketitle

\begin{abstract}
We prove that an additive form of degree $d=2m$, $m$ odd over any totally ramified extension of $\mathbb{Q}_2$ has a nontrivial zero if the number of variables $s$ satisifies $s \ge \frac{d^2}{4} + 3d + 1$.
\end{abstract}

\section{Introduction}
Consider additive forms of degree $d$ over a p-adic field $K$ in $s$ variables:
\begin{equation}
\label{form}
a_1x_1^{d} + a_2x_2^{d} + \ldots + a_{s}x_{s}^{d}.
\end{equation}
Let $\Gamma^*(d, K)$ represent the minimum number such that any additive form of degree $d$ over $K$ in at least that many variables has a nontrivial zero.
A refinement of a famous conjecture by Artin holds that $\Gamma^*(d, K) \le d^2 + 1$ for any $p$-adic field $K$.
The seminal paper in this direction by Davenport and Lewis \cite{davenport1963homogeneous} showed that the bound holds for all fields of $p$-adic numbers $\mathbb{Q}_p$, and that equality holds when $d = p - 1$.
In fact, a similar argument gives an additive form in $d^2$ variables with no nontrivial zero for any totally ramified $p$-adic field and $d = p - 1$.

Aside from combinations of degrees and fields which make the problem relatively trivial, this bound of $d^2+1$ has been extended to proper extensions of $\mathbb{Q}_p$ in only a few restricted cases:
In \cite{leep2018diagonal} it was shown that $\Gamma^*(d, K) \le d^2+1$ for $K$ any unramified extension of $\mathbb{Q}_p$ with $p$ odd.
It was demonstrated in \cite{duncan2022quarticsolubility} that $\Gamma^*(4, K) \le d^2+1$ for any $K$ of the four (out of the six total) ramified quadratic extensions $\mathbb{Q}_2(\sqrt{2})$, $\mathbb{Q}_2(\sqrt{-2})$, $\mathbb{Q}_2(\sqrt{10})$, and $\mathbb{Q}_2(\sqrt{-10})$.
And, most recently it was proven in \cite{MGK} that $\Gamma^*(d, K) \le d^2 + 1$ for $K$ any of the seven quadratic extensions of $\mathbb{Q}_2$ and $d$ not a power of 2.
Note that in all of these cases the degree of ramification is at most two.  In this paper, I will show that the bound holds for an infinite family of degrees over any totally ramified extension of $\mathbb{Q}_2$ of arbitrarily large degree.  In fact (aside from the case $d=2$, which is known by other means) I will show that a bound \textit{strictly lower} than $d^2+1$ holds in all of these cases.

\begin{theorem}
\label{main}
Let $d=2m$, with $m$ any odd number at least 3, and $K$ be any totally ramified extension of $\mathbb{Q}_2$.  Then,
$$\Gamma^*(d, K) \le \frac{d^2}{4} + 3d + 1.$$
\end{theorem}

\section{Preliminaries}

Let $K$ denote a totally ramified extension of $\mathbb{Q}_2$, $\mathcal{O}$ denote its ring of integers, and $e$ denote its degree of ramification, so that $2=u\pi^e$ where $u$ is a unit in $\mathcal{O}$.
Without loss of generality, assume $a_i \in \mathcal{O}\backslash\{0\}$.
Let $\pi$ be a uniformizer (generator of the unique maximal ideal) of $\mathcal{O}$.
Any $a \in \mathcal{O}$ can be written $a = \pi^\ell(a_0 + a_1\pi + a_2\pi^2 + a_3\pi^3 + \ldots)$, with $a_i \in \{0,1\}$ and $a_0 \ne 0$.
I will refer to $\ell$ as the \textit{level} of the variable.

By the change of variables $\pi^r x^d = \pi^{r-id}(\pi^i x)^d = \pi^{r-id}y^d$ for $i \in \mathbb{Z}$, we will consider the level of a variable modulo $d$ for the remainder of the paper.  Multiplying a form by $\pi$ increases the level of each variable by one, and does not affect the existence of a nontrivial zero.  Considering the levels of variables modulo $d$, applying this transformation any number of times effects a cyclic permutation of the levels.
This is useful for arranging the variables in an order which is more convenient, a process to which we will refer as \textit{normalization}. See Lemma 3 of \cite{davenport1963homogeneous} for a proof of the following Lemma.

\begin{lemma}
Given an additive form of degree $d$ in an arbitrary local field $K$, let $s$ be the total number of variables, and $s_i$ be the number of variables in level $i \pmod{d}$. By a change of variables, the form may be transformed to one with:
\begin{align}
\begin{split}
s_0 \ge \frac{s}{d},  
\end{split}
\begin{split}
s_0 + s_1 \ge \frac{2s}{d},
\end{split}
\begin{split}
\ldots,
\end{split}
\begin{split}
s_0 + \ldots + s_{d-1} = s.
\end{split}
\end{align}
\end{lemma}
Consider a collection of terms $a_{i_1}x_{i_1}^d + a_{i_2}x_{i_2}^d + \ldots + a_{i_n}x_{i_n}^d$.
Some assignment $x_{i_j} = b_{i_j}$ of values to the variables yields some value $b$.
If instead we assign $x_{i_j} = b_{i_j}y$, then the terms are replaced with the new term $by^d$.
I call this replacement a \textit{contraction} of the original variables.
The level of the resulting variable is at least as high as the lowest level of the variables used in the contraction.
I proceed by showing that given a configuration of variables enough contractions can be performed so that a variable is produced which is at a sufficiently high level relative to the levels of the variables used to produce it.
A nontrivial zero then follows from the following version of Hensel's Lemma specialized to additive forms of degree $d=2m$ over totally ramified extensions of $\mathbb{Q}_2$.  (For a proof, see \cite{leep2018diagonal}.)

\begin{lemma}[Hensel's Lemma]
\label{quartic_hensels_lemma}
Let $x_i$ be a variable of (\ref{form}) at level $k$.
Suppose that $x_i$ can be used in a contraction of variables (or one in a series of contractions) which produces a new variable at level at least $2e+1$.
Then (\ref{form}) has a nontrivial zero.
\end{lemma}

As a result, we will henceforth assume whenever possible that a contraction or series of contractions produces a variable at most $2e$ levels higher than the variables used in the contractions.

Consider two variables in the same level with coefficients $a = \pi^\ell(a_0 + a_1\pi + \ldots)$ and $b = \pi^\ell(b_0 + b_1\pi + \ldots)$ with $a_i, b_i \in \{0,1\}$.
Because $a_0$ and $b_0$ are nonzero, they must both be $1$, so $a + b = 2 + (a_1 + b_1)\pi + \ldots = (a_1 + b_1)\pi + \ldots + \pi^e + \ldots$, and their sum is divisible by $\pi$.
Thus, any two variables can be contracted at least one level higher.

Note that $(1 + \pi^k)^d = (1 + 2\pi^k + \pi^{2k})^m = (1 + u\pi^{k+e} + \pi^{2k})^m$.
If $k < e$, then $\pi^{2k}$ is a lower order term than every term of $u\pi^{k+e}$, and so $(1 + \pi^k)^d = (1 + \pi^{2k} + \ldots)^m = 1 + m\pi^{2k} + \ldots \equiv 1 + \pi^{2k} \pmod{\pi^{2k+1}}$.
Suppose a contraction involving a term $a_i x_i$ in level $\ell$ results in a variable in level $\ell + 2k$ with $k < e$.
That is, there is some substitution for the variables such that the resulting sum is $c = \pi^{\ell + 2k}(1 + c_1\pi + \ldots)$.
Replacing the substitution $\beta_i$ of $x_i$ with $\beta_i(1+\pi^k)$, the resulting sum is $c + a_i \pi^{2k} = \pi^{\ell + 2k}(1 + c_1\pi + \ldots) + \pi^{\ell + 2k}(1 + a_{i_1}\pi + \ldots) \equiv 0 \pmod{\pi^{\ell + 2k + 1}}$.
Thus the contraction which would have a produced a variable $2k$ levels higher can be arranged so that the new variable is instead produced at least $2k+1$ levels higher.
Similarly, if the contraction would have gone up at least $2k+1$ levels, it can be arranged to go up exactly $2k$ levels, for any $k<e$.
I will refer to this ability to make a contraction involving a variable from level $\ell$ bypass or stop at any of the levels $\ell+2, \ell+4, \ldots, \ell+2e-2$ as the resulting variable being \textit{free} in those levels.
Note that if a variable which is free in some levels is used in a subsequent contraction, then the new resulting variable is also free in those levels (as well as new levels arising from this contraction) as the earlier contraction can be arranged to produce an additional term of the appropriate order without affecting any terms of lower order.

\section{Proof}

\begin{lemma}
\label{bins}
Let every unordered pair among $n$ objects be assigned to one of $m$ bins.  If $n\ge m+3$, then two disjoint pairs get assigned to the same bin.
\end{lemma}
\begin{proof}
Suppose there is an assignment of pairs of $m+3$ objects to $m$ bins which does not assign two disjoint pairs to the same bin.  If a bin contains two pairs, then by hypothesis they must have an object in common, say $(a,b)$ and $(a,c)$.  If that bin contains an addition pair, by hypothesis it is either of the form $(b,c)$ or $(a,d)$.  If it is of the form $(b,c)$, then any addition pair would be disjoint to one of the pairs already assigned to that bin.  If it is of the form $(a,d)$, then any additional pair which does not contain the common object $a$ would be disjoint to one of the pairs already assigned to that bin.  In the former case, the bin contains at most 3 pairs, and in the latter case it contains at most $m+2$ pairs.

Consider now the number of pairs in each bin.
Start with one of these bins with the maximum number of pairs; it contains at most $m+2$ pairs and has some common object, say $a$.
Next, take a bin with the maximum number of pairs among the remaining bins;  it has some common object $b$.
If the previous bin contained the pair $(a,b)$, then this bin does not, and contains at most $m+1$ pairs.
If the previous bin did not contain $(a,b)$, then it contained at most $m+1$ pairs and thus the present bin also contains at most $m+1$ pairs.
Continuing in this way, note that the next bin contains at most $m$ pairs, and so on, and so the kth bin considered contains at most $m + 3 - k$ pairs.
Thus we can conclude that the number of pairs contained in all bins is no more than $(m+2) + (m+1) + \ldots + 3 = \binom{m+3}{2} - 3$.
\end{proof}

\begin{lemma}
\label{m+7}
Suppose any level contains at least $m+7$ variables.  Then (\ref{form}) has a nontrivial zero.
\end{lemma}
\begin{proof}
Without loss of generality, assume the level containing the $m+7$ variables is level 0.  Contracting a pair of variables from level 0 results in a variable that is free in levels $2, 4, \ldots, 2e-2$.  Thus we may assume the new variable is produced in an odd numbered level or in level $2e$.  If there are any pairs which contract to level $2e$, first perform these contractions.  If there are at least two such pairs, then the two resulting variables can be contracted to at least level $2e+1$ and a zero follows from Hensel's Lemma.  Now, consider all the pairs which contract to odd numbered levels.  There are $m$ such levels and at least $m+3$ variables with which to form pairs.  By \Cref{bins}, two pairs contract to the same odd level $k$.  Contracting this pair results in a variable which is free in levels $2, 4, \ldots, 2e-2$ and $k+2, k+4, \ldots, 2e-1$, and so can be contracted to level $2e$.  If a variable in level $2e$ was produced in the initial phase, it can be contracted with this new variable to level at least $2e+1$ and a zero follows from Hensel's Lemma.  Otherwise, there remain at least $m+3$ variables in level 0, and so by \Cref{bins}, another variable can be produced in level $2e$.  A zero follows as before.
\end{proof}

\begin{lemma}
\label{m+3_2}
Suppose level $k$ contains at least $m+3$ variables and level $k+1$ contains at least $2$ variables.  Then (\ref{form}) has a nontrivial zero.
\end{lemma}
\begin{proof}
Without loss of generality, assume $k=0$.  As in \Cref{m+7}, either a pair from level 0 can be contracted to level $2e$, or two pairs can be contracted to the same odd level, and from there to level $2e$ by \cref{bins}.  Now, contract a pair from level 0 and level 1 to variables $y_0$ and $y_1$ at levels $i$ and $j$ respectively.  We may assume by the free levels of the variables that $i$ is odd or $2e$ and $j$ is even or $2e+1$.  Further, if $i=2e$, then $y_0$ can be contract at least one level higher.  Thus assume $i$ is odd.  If $i$ is less than $j$, then because $y_1$ is free at level $j$, we can arrange the contraction so that it goes to level $j$ instead.  The two variables in level $j$ can be contracted, resulting in a variable which is free at every level lower than $2e$, the variable is then produced in $2e$, and a zero follows as before.  Thus assume $j < i$.  Then $y_0$ is free in level $j$ and a zero follows as before.
\end{proof}

\begin{lemma}
Suppose level $k$ contains at least $4$ variables and level $k+1$ contains at least $4$ variables.  Then (\ref{form}) has a nontrivial zero.
\end{lemma}
\begin{proof}
Contract a pair from level 0 and level 1 to variables $y_0$ and $y_1$ at levels $i$ and $j$ respectively.  We may assume by the free levels of the variables that $i$ is odd or $2e$ and $j$ is even or $2e+1$.  If $i=2e$, proceed to the next step.  Otherwise, assume $i$ is odd.  If $i$ is less than $j$, then because $y_1$ is free at level $j$, we can arrange the contraction so that it goes to level $j$ instead.  The two variables in level $j$ can be contracted, resulting in a variable which is free at every level lower than $2e$, the variable is then produced in $2e$, and a zero follows as before.  Thus assume $j < i$.  Then $y_0$ is free in level $j$ and a zero follows as before.

Repeat this process again, producing a second variable at level $2e$.  A zero follows.
\end{proof}

If the form contains at least $\frac{d^2}{4} + 3d + 1$, then by normalization it falls under one of the three lemmas above, completing the proof of \Cref{main}.

\bibliographystyle{plain}
\bibliography{biblio}

\begin{thebibliography}{1}

\bibitem{davenport1963homogeneous}
H.~Davenport and D.~J. Lewis.
\newblock Homogeneous additive equations.
\newblock {\em Proc. Roy. Soc. London Ser. A}, 274:443--460, 1963.

\bibitem{MGK}
Bruno de~Paula~Miranda, Hemar Godinho, and Michael~P. Knapp.
\newblock Diagonal forms over quadratic extensions of {$\mathbb{Q}_2$}.
\newblock {\em Publ. Math. Debrecen, to appear}.

\bibitem{duncan2022quarticsolubility}
Drew Duncan and David~B Leep.
\newblock Solubility of additive quartic forms over ramified quadratic
  extensions of $\mathbb{Q}_2$.
\newblock {\em International Journal of Number Theory}, 2022.

\bibitem{leep2018diagonal}
David~B. Leep and Luis Sordo~Vieira.
\newblock Diagonal equations over unramified extensions of {$\Bbb Q_p$}.
\newblock {\em Bull. Lond. Math. Soc.}, 50(4):619--634, 2018.

\end{thebibliography}

\end{document}